\documentclass[dvipsnames]{amsart}
\usepackage{graphicx} 
\usepackage{amssymb}
\usepackage{geometry}
\usepackage{tikz}
\usetikzlibrary{calc}
\usetikzlibrary{lindenmayersystems, backgrounds}

\pgfdeclarelindenmayersystem{cayley}{
	\rule{F -> F [ R [F] [+F] [-F] ]}
	\symbol{R}{
		\pgflsystemstep=0.5\pgflsystemstep
	} 
}

\newcommand{\nc}{\newcommand}
\nc{\dmo}{\DeclareMathOperator}
\nc{\rnc}{\renewcommand}
\nc{\nt}{\newtheorem}

\newtheorem{theorem}{Theorem}[section]
\newtheorem{lemma}[theorem]{Lemma}
\nt{proposition}[theorem]{Proposition}
\nt{question}{Question}

\nc{\B}{B}
\nc{\inv}{^{-1}}
\nc{\C}{\mathbb C}
\nc{\R}{\mathbb R}
\nc{\Z}{\mathbb Z}
\nc{\s}{\sigma}

\DeclareMathOperator{\Aut}{Aut}
\DeclareMathOperator{\GL}{GL}
\DeclareMathOperator{\Sp}{Sp}

\DeclareMathOperator{\CRS}{CRS}
\DeclareMathOperator{\Mod}{Mod}
\dmo{\Stab}{Stab}

\nc{\normalclosure}[1]{\left\langle\hspace{-.7mm}\left\langle{#1}\right\rangle\hspace{-.7mm}\right\rangle}

\nc{\p}[1]{\bigskip\noindent\emph{#1.}}

\begin{document}

\title{Totally Symmetric Sets}


%    author one information
\author{Noah Caplinger}
\address{Noah Caplinger \\ Department of Mathematics \\  University of Chicago \\ 5734 University Ave.\\Chicago, IL 60637}
\thanks{}
\email{nccaplinger@gmail.com}

%    author two information
\author{Dan Margalit}
\address{Dan Margalit \\ Department of Mathematics\\ Vanderbilt University \\ 1326 Stevenson Center Ln \\ Nashville, TN 37240}
\email{dan.margalit@vanderbilt.edu}
\thanks{This material is based upon work supported by the National Science Foundation under Grant No.\ DMS-2203431.}

%    The 2020 edition of the Mathematics Subject Classification is
%    the current definitive version.

\date{\today}

\begin{abstract}
We survey the theory of totally symmetric sets, with applications to homomorphisms of symmetric groups, braid groups, linear groups, and mapping class groups.  
\end{abstract}

\dedicatory{Dedicated to Mike Mihalik on the occasion of his retirement}

\maketitle

\section{Introduction}

The theory of totally symmetric sets is a tool that has been proven to be useful in classifying homomorphisms between certain types of groups.  The basic definitions were introduced by Kordek and the second author in their study of homomorphisms between braid groups \cite{Kordek-Margalit}.  Since that work, the theory has been used in the study of homomorphisms between symmetric groups, braid groups, linear groups, and mapping class groups.

Here is the definition.  A totally symmetric set in a group $G$ is a subset
\[
X = \{x_1,\ldots,x_k\} \subset G
\]
with the following property: for every $\sigma \in \Sigma_k$, there is a $g_\sigma \in G$ such that
\[
g_\sigma x_i g_\sigma^{-1} = x_{\sigma(i)}
\]
for all $i$.  It follows from the definition that the elements of a totally symmetric set lie in a single conjugacy class.

Evidently, if $X \subseteq G$ is a totally symmetric set and $f : G \to H$ is a homomorphism, then $f(X)$ is a totally symmetric set in $H$.  As we show in Section~\ref{sec:tss}, a much stronger condition is true: $f(X)$ is either a totally symmetric set of size $|X|$ or it is a singleton.  In the phrasing of Salter and the first author: collision implies collapse.  This is the fundamental property of totally symmetric sets.

For groups $G$ and $H$, the collision-implies-collapse property yields an (unreasonably effective) blueprint for classifying homomorphisms $f : G \to H$, as follows:
\smallskip
\begin{quote}
\begin{itemize}
    \item[\emph{Step 1.}] Find a large totally symmetric set $X \subset G$.
    \item[\emph{Step 2.}] Classify the large totally symmetric sets in $H$.
    \item[\emph{Step 3.}] Deduce properties of $f(X)$ and draw conclusions about $f$.
\end{itemize}
\end{quote}
\smallskip
For instance, if $G$ has a totally symmetric set $X$ with $|X| = k$, and $H$ has no totally symmetric set of cardinality $k$, then any homomorphism $f: G \to H$ must collapse $X$.  Moreover, for any $x_i,x_j \in X$ the normal closure of $x_ix_j^{-1}$ lies in the kernel of $f$.  In particular, if the $x_ix_j^{-1}$ are normal generators for $G$, then $f$ is the trivial map.  

Step 2 of the blueprint is generally the most challenging (and interesting).  For a given group $H$, the approach is to choose a space $Y$ on which $H$ acts.  As explained in Section~\ref{sec:tsc}, totally symmetric sets in $H$ correspond to totally symmetric geometric configurations in $Y$.  These configurations can be, for example, configurations of eigenspaces for linear maps or canonical reduction systems for mapping classes.  The main strategy is to classify the totally symmetric configurations, and then use this to classify the totally symmetric sets.

\p{Overview} In Section~\ref{sec:tssbp} we introduce the basic notions and examples in the theory of totally symmetric sets.  We also prove the collision-implies-collapse property.  We then use the blueprint to give an intuitive explanation for why the outer automorphism group of the symmetric group is (usually) trivial.

In Sections~\ref{sec:gl} and~\ref{sec:braid} we explain how the blueprint is applied in the cases of the general linear group and the braid group.  The former case was addressed in a paper by the first author and Salter \cite{Caplinger-Salter} and the latter in a paper by the second author and Kordek \cite{Kordek-Margalit}.  As per the blueprint, the strategy in both cases is to classify large totally symmetric configurations and then to promote this to a classification of large totally symmetric sets.  For the general linear group, the configurations are configurations of subspaces.  We then use the theory of Jordan normal form to do the promotion.  For the braid group the configurations are configurations of multicurves.  In this case we use Nielsen--Thurston theory to do the analogous promotion.  

In Section~\ref{sec:kolay} we prove a theorem of Kolay, which says that the standard map $\B_n \to \Sigma_n$ gives the smallest non-cyclic quotient of the braid group.  To streamline Kolay's argument, we first introduce a variation on totally symmetric sets, namely, collapsing sets.  These are exactly the sets that satisfy the collision-implies-collapse property.  We then present Kolay's proof of the theorem.  

Finally, in Section~\ref{sec:spec} we make an explicit analogy between the collision-implies-collapse property and Schur's lemma from representation theory.  We discuss several questions that arise from this analogy.

\p{Acknowledgments} We are grateful to Michael Griffin, Koichi Oyakawa,  Larry Rolen, Shuxian Song, Isabelle Steinmann, and an anonymous referee for helpful comments and conversations.  We would also like to thank Lei Chen, Kevin Kordek, Justin Lanier, Dan Minahan, and Nick Salter for enlightening conversations about totally symmetric sets.

%%%
%%%
%%%

\section{Totally symmetric sets and the blueprint}
\label{sec:tssbp}

The three subsections in this section correspond to the three steps of the blueprint for totally symmetric sets.  In Section~\ref{sec:tss} we give some basic examples of totally sets, and state and prove the collision-implies-collapse lemma.  In Section~\ref{sec:tsc}, we define totally symmetric configurations, and use them to give upper bounds on the sizes of totally symmetric sets in certain groups.  Finally, in Section~\ref{sec:sn} we use the results of Sections~\ref{sec:tss} and~\ref{sec:tsc} to give a conceptual explanation for the classification of automorphisms of the symmetric group $\Sigma_n$.  

\subsection{Totally symmetric sets and collision implies collapse}
\label{sec:tss}

\p{Examples of totally symmetric sets} Among the most basic examples of totally symmetric sets are:
\begin{align*}
    \{(1 \ 2),  (3 \ 4), \ldots\} &\subseteq \Sigma_n,  \ \  \{(1 \ i), \dots, (1 \ n)\} \subseteq \Sigma_n,  \ \ 
    \{ E_{1,2}, \dots, E_{1,n} \} \subseteq \GL_n(\Z),   \\
    &\{ E_1, \dots, E_n \} \subseteq \GL_n(\Z),\ \text{ and} \ 
    \{\sigma_1, \sigma_3,\ldots\} \subseteq \B_n.  
\end{align*}
where the $(i \ j)$ are transpositions in the symmetric group $\Sigma_n$, the $\sigma_i$ are the standard half-twist generators for the braid group $\B_n$, the $E_{i,j}$ are elementary matrices in the general linear group $\GL_n(\Z)$, and the $E_i$ are the elements of $\GL_n(\Z)$ obtained from the identity by negating the $i$th diagonal entry.  We leave it as an exercise to verify that these are all totally symmetric sets.   

\p{Collision implies collapse} The following lemma, mentioned in the introduction, encapsulates the fundamental property of totally symmetric sets.  The lemma originally appears in the work of Kordek and the second author of this paper \cite[Lemma 2.1]{Kordek-Margalit}.

\begin{lemma}
\label{cic}
Let $f : G \to H$ be a homomorphism of groups.  If $X \subseteq G$ is a totally symmetric set then $f(X)$ is either a totally symmetric set of cardinality $|X|$ or a singleton.  \end{lemma}

\begin{proof}

Let $x,y,z \in X$ and suppose $f(x) = f(y)$. Total symmetry guarantees some $g \in G$ such that $g (x,y,z) g^{-1} = (x,z,y)$.  We have:
\[
f(xz\inv) = f(g x y\inv g\inv) = f(g)f(x y\inv)f(g)^{-1} = 1.
\]
Thus $f(z) = f(x)$ and the lemma follows. 
\end{proof}

In the original paper by Kordek and the second author, totally symmetric sets were assumed to have the additional property that the elements commute pairwise.  So in that paper, the set $\{(1 \ 2), \; (3 \ 4), \ldots\}$ would be considered as a totally symmetric set in $\Sigma_n$, whereas $\{(1 \ i), \dots, (1 \ n)\}$ would not.  The commutativity condition was included because it simplifies the classification of totally symmetric configurations for braid groups.  Since the more general totally symmetric sets considered here still satisfy Lemma~\ref{cic}, we will henceforth use the term ``commuting totally symmetric set'' to refer to a totally symmetric set with the additional property that the elements commute pairwise.  

In defining totally symmetric sets, Kordek and the second author were directly inspired by the work of Aramayona--Souto, who used a symmetric group action on a collection of Dehn twists to similar effect \cite[Section 5]{AS}.  

%%%
%%%
%%%

\subsection{Totally symmetric configurations and upper bounds on totally symmetric sets}
\label{sec:tsc}

Recall that Step 2 of the blueprint concerns the classification of large totally symmetric sets in a given group $G$ (in the blueprint the group is called $H$).  After the fashion in geometric group theory, we approach this problem by considering a suitable action of $G$ on a space $Y$.  Given such an action, we often have at least one natural choice of function:
\begin{align*}
G &\to \text{subsets of } Y \\
g &\mapsto Y_g
\end{align*}
For a given $g \in G$, the subset $Y_g$ might be the fixed set, or an eigenspace, or an invariant axis, etc.  As long as the association $g \mapsto Y_g$ is natural, it will satisfy the equation
\[
Y_{hgh^{-1}} = h \cdot Y_g
\]
for all $h \in G$ (we can take this equivariance condition as the definition of naturality).  In particular, if $\{x_1,\dots,x_k\}$ is a totally symmetric subset of $G$, then $\{Y_{x_1},\dots,Y_{x_k}\}$ is a totally symmetric configuration in $Y$ in the sense that any permutation of the $Y_{x_i}$ can be realized by the action of $G$.

\p{Totally symmetric configurations} Motivated by this discussion, we can give the definition of a totally symmetric configuration.  
Suppose that a group $G$ acts on a space $Y$.  Let
\[
\{Y_1, \dots, Y_k \}
\]
be a collection of subspaces of $Y$.  We say that $\{Y_i\}$ is a totally symmetric configuration if for each $\sigma \in \Sigma_k$, there is a $g_\sigma \in G$ so that
\[
g_\sigma \cdot (Y_i) = Y_{\sigma(i)}
\]
for all $i$.  Again, the point of the definition is that, as long as the association of a subspace to a group element satisfies the naturality property $Y_{hgh^{-1}} = h \cdot Y_g$, the configuration associated to a totally symmetric set is a totally symmetric configuration.

%%%
%%%
%%%

\p{Unifying the definitions} Our definitions of totally symmetric sets and totally symmetric configurations are almost identical.  As observed by the first author and Salter \cite[Definition 2.1]{Caplinger-Salter}, they can be combined into one definition as follows: 
\begin{quote}
\emph{
Let $G$ act on a set $Z$. A subset $X = \{x_1,\ldots,x_k\}\subset Z$ is totally symmetric if for all $\sigma \in \Sigma_k$, there is some $g_\sigma \in G$ such that $g_\sigma \cdot x_i = x_{\sigma(i)}$
}
\end{quote}
The definition of a totally symmetric set is recovered by considering the action of $G$ on itself by conjugation, and the definition of a totally symmetric configuration is recovered by considering the action of $G$ on a set of subsets of a space $Y$ that carries an action of $G$.  Even in this general setting, totally symmetric sets obey the collision-implies-collapse principle where the homomorphism in Lemma~\ref{cic} is replaced by a $G$-equivariant map.

\p{Example: Dihedral groups} We will use the notion of totally symmetric configurations to prove the following fact:
\begin{quote}
\emph{A totally symmetric set $X \subseteq D_{n}$ has $|X| \leq 3$.}
\end{quote}
The first step is to prove that a totally symmetric set of rotations has cardinality at most two (exercise).  Suppose then that $X$ is a totally symmetric set consisting of reflections. To each reflection in $X$ we can associate the corresponding line of reflection in the plane.  As above, this gives a totally symmetric configuration of lines in the plane.  The largest such configuration has three lines (another exercise).  Since reflections are determined by the corresponding lines, the desired statement follows.

This argument shows more: 
\begin{quote}
\emph{If $X$ is a totally symmetric set of $D_n$ with $|X|=3$ then $3 \mid n$ and $X$ consists of reflections about lines that pairwise form an angle of $\pi/3$.}
\end{quote}
As a sample consequence, we have the following fact:
\begin{quote}
    \emph{If $n \geq 8$ and $m \geq 3$, then every homomorphism $\B_n \to D_m$ has cyclic image.}
\end{quote}
While this fact is not difficult to prove directly, the theory of totally symmetric sets gives a natural, conceptual explanation.  

\p{Example: Free groups} We now use totally symmetric configurations to prove the following:
\begin{quote}
\emph{A totally symmetric set in $X \subseteq F_2$ has $|X| \leq  1$.}
\end{quote}

Consider the action of $F_2$ on its Cayley graph, the regular four-valent tree $T_4$.  Each element $x \in F_2$ acts on $T_4$ by translating along an axis, which is a bi-infinite geodesic in $T_4$.  Again, this gives a totally symmetric configuration of geodesics.  

Within a given conjugacy class in $F_2$, an element is determined by its axis.  Therefore, it suffices to show that there is no totally symmetric configuration consisting of two bi-infinite geodesics in $T_4$.  

Let $Y \subset T_4$ be a totally symmetric configuration of bi-infinite geodesics.  If $Y_1$ and $Y_2$ are distinct elements of $Y$, then by total symmetry there is an element of $F_2$ interchanging $Y_1$ and $Y_2$.  This is impossible, since (using the usual embedding of $T_4$ in the plane)  the elements of $F_2$ act  on $T_4$ by orientation-preserving planar automorphisms.  

\p{An upper bound for all groups} The first author proved the following result \cite[Theorem 1]{CaplingerSn}, which gives an upper bound for the cardinality of a totally symmetric set in an arbitrary group.  

\begin{theorem}[Caplinger]
\label{thm:bound}
Let $X$ be a totally symmetric set in a group $G$.  If $|X| \geq 4$, then 
\[
|G| \geq (|X|+1)!
\]
Equality is attained only when $G = \Sigma_n$.
\end{theorem}

As a sample application, any totally symmetric set in the monster group $M$ has cardinality less than 44, since $44! > |M|$.  

Analogous (but not sharp) upper bounds on the cardinalities of commuting totally symmetric sets were proved by Chudnovsky--Kordek--Li--Partin \cite[Proposition 2.2]{reu} and by Scherich--Verberne \cite[Theorem A]{sv}.  

\p{Other upper bounds on commuting totally symmetric sets} Kordek--Li--Partin \cite{UpperBoundsTSS} provide a suite of upper bounds for the cardinality of a commutative totally symmetric set.  For instance they show that the largest cardinality of a commutative totally symmetric set in the dihedral group is 2 \cite[Theorem 3.3]{UpperBoundsTSS} and that the largest cardinality of a commutative totally symmetric set in the Baumslag--Solitar group $BS(1,n)$ with $n \neq -1$ is 1. They also prove that the largest cardinality of a commutative totally symmetric set in a product $G \times H$ or $G \ast H$ is the supremum of the cardinalities for totally symmetric sets in a single factor.  They also prove that a solvable group cannot have a commutative totally symmetric set with 5 elements.  We refer the reader to their paper for a full accounting of their results.

%%%
%%%
%%%

\subsection{Application to the symmetric group} 
\label{sec:sn}

Using the basic theory of totally symmetric sets already established, we can give a conceptual explanation of the following classical theorem.  This argument originally appeared in the work of the first  author \cite{CaplingerSn}.

\begin{theorem}
\label{thm:autsn}
For $n \geq 7$, the outer automorphism group of $\Sigma_n$ is trivial.
\end{theorem}
Let $Z_n$ denote the totally symmetric set
\[
Z_n = \{(1 \ i) \mid i \geq 2\} \subseteq \Sigma_n. 
\]
To prove Theorem~\ref{thm:autsn} we use the following auxiliary result \cite[Theorem 2]{CaplingerSn}, which is a classification of large totally symmetric sets in $\Sigma_n$:
\begin{quote}
\emph{Let $n \geq 7$.  If $X \subset \Sigma_n$ is a totally symmetric set with $|X| \geq n-1$. Then $X$ is conjugate to $Z_n$. }\end{quote}
From this fact, the proof of Theorem~\ref{thm:autsn} proceeds as follows.  Let $f : \Sigma_n \to \Sigma_n$ be an automorphism. Then $f(Z_n)$ is equivalent to $Z_n$ in the following sense: there exists  $\tau \in \Sigma_n$ with $\tilde{\tau}f(Z_n) = Z_n$, where $\tilde{\tau}$ is the inner automorphism corresponding to $\tau$. Then $\tilde{\tau} \circ f$ permutes $Z_n$, so total symmetry gives some $\sigma \in \Sigma_n$ so that 
\[
\tilde{\sigma}\circ \tilde{\tau}\circ f = \widetilde{\sigma\tau} \circ f
\]
is the identity on $Z_n$.  Since $Z_n$ generates $\Sigma_n$, we conclude that $\widetilde{\sigma\tau} \circ f = \textrm{id}$, that is, $f = (\widetilde{\sigma\tau})^{-1}$.  In particular, $f$ is an inner automorphism, completing the proof of the theorem.

A slight modification of the above argument yields the following (well-known) generalization of Theorem~\ref{thm:autsn}.

\begin{theorem}
Let $7 \leq n \leq m$, and let $f : \Sigma_m \to \Sigma_n$ be a homomorphism whose image is not cyclic.  Then $m=n$ and $f$ is an inner automorphism.
\end{theorem}

To obtain this stronger theorem, the only additional observation required is that---by the classification of large totally symmetric sets in $\Sigma_n$---the restriction $f|Z_n$ cannot be injective when $n < m$.  Thus it must be trivial, and so the image is cyclic (of order at most 2).  

\medskip

We would be remiss not to describe the situation for $\Sigma_6$, which does have a nontrivial outer automorphism.  From the perspective of totally symmetric sets, the reason why this outer automorphism exists is that $\Sigma_6$ has two conjugacy classes of totally symmetric sets with five elements: the standard one and its image under the nontrivial outer automorphism.  

\medskip

The proof of Theorem~\ref{thm:autsn} given here is not simpler than the classical proof.  However, it gives a conceptually simple, structural explanation.  Also, the classification of large totally symmetric sets creates a broad tool for studying any homomorphism to or from $\Sigma_n$.  We will return to this theme several times in what follows.

%%%
%%%
%%%

\section{Totally symmetric sets in the general linear group}
\label{sec:gl}

In this section we turn our attention to the general linear group.  The first author and Salter give a classification of large totally symmetric sets in $\GL_n(\C)$, Theorem~\ref{thm:linear_tss} below.  They used this classification to give a new, conceptual proof of the following classical fact:
\begin{quote}
\emph{Any non-abelian representation of $\Sigma_n$ has dimension at least $n~-~1$.}
\end{quote}
We will start by describing the largest totally symmetric sets in $\GL_n(\C)$, then state the classification theorem, and then explain the applications to representation theory.  Following the blueprint, we then classify large totally symmetric configurations in $\C^n$, before using this classification to prove the classification of totally symmetric sets in $\GL_n(\C)$.

\p{Standard totally symmetric sets} Consider a regular $n$-simplex $\Delta \subset \R^n$ centered at the origin. The set of vertices $v_1,\ldots,v_{n+1}$ of $\Delta$ is a totally symmetric configuration in the sense that for all $\sigma \in \Sigma_{n+1}$, there is an $A_\sigma$ in $\GL_n(\C)$ (in fact in $O(n)$) such that $A_\sigma v_i = v_{\sigma(i)}$. The hyperplanes $v_i^\bot$ also form a totally symmetric configuration with the same choice of $A_\sigma$. This allows us to form a totally symmetric set in $\GL_n(\C)$ as follows. Pick $\lambda, \mu \in \C$ distinct and non-zero, and define $A_i \in \GL_n(\C)$ by declaring $\C v_i$ and $v_i^\bot$ to be $\lambda$- and $\mu$-eigenspaces respectively. The set $\mathcal{A}_n = \{A_1,\ldots,A_{n+1}\}$ is then totally symmetric.  We refer to (any conjugate of) any such $\mathcal{A}_n$ as a standard totally symmetric set in $\GL_n(\C)$.  

\p{The classification of large totally symmetric sets} The following theorem says that the above construction is the only construction of totally symmetric sets in $\GL_n(\C)$ of cardinality $n+1$.

\begin{theorem}[Caplinger--Salter]
\label{thm:linear_tss}
Let $n\neq 5$ and let $X\subset \GL_n(\C)$ be a totally symmetric set. Then $|X|\leq n+1$, and equality is achieved exactly when $X$ is standard.
\end{theorem}

This theorem immediately applies to bound the dimension of a faithful representation of any group: if $G$ contains a totally symmetric subset of size $n$, then $G$ has no faithful representations in dimension less than $n-1$.  In the case of $\Sigma_n$, we can say more.

\p{Application to representations of the symmetric group} Building on the last idea, let $\rho~:~\Sigma_n\to \GL_m(\C)$ be a non-abelian representation of $\Sigma_n$.  We would like to show $m \geq n-1$.  

For the standard totally symmetric set $Z_n$ in $\Sigma_n$ we have that $\rho(Z_n) \subset \GL_m(\C)$ is a totally symmetric set.  We show that if $m < n-1$ then $\rho(Z_n)$ is a singleton, implying that $\rho$ has cyclic image.  We treat the cases $m < n-2$ and $m=n-2$ in turn.

If $m < n-2$ then Theorem~\ref{thm:linear_tss} gives that there is no totally symmetric subset of $\GL_m(\C)$ of size $n-1$.  Thus by Lemma~\ref{cic}, $\rho(Z_n)$ is a singleton, as desired.  

If $m=n-2$, Theorem~\ref{thm:linear_tss} gives that $\rho(Z_n)$ is a singleton or is standard.  But no two distinct elements of the standard totally symmetric set $\mathcal{A}_n$ satisfy the braid relation, so $\rho(Z_n)$ must be a singleton, as desired.

\p{Application to algebraic geometry} Let $\mathcal{U}_{n,d}$ denote the space of smooth degree $d$ hypersurfaces in $\C\mathbb{P}^n$.  We can continue the above line of reasoning in order to constrain the dimension of a representation of $\pi_1(\mathcal{U}_{n,d})$:
\begin{quote}
\emph{
If $\rho: \pi_1(\mathcal{U}_{n,d}) \to \GL_m(\C)$ is a non-cyclic representation with $d\geq 5$, then
\[
m \geq \left\lceil \frac{d-1}{2}\right \rceil^n -1.
\]
}
\end{quote}
We now outline the proof of this fact.  Lönne gave a presentation of $\pi_1(\mathcal{U}_{n,d})$ generalizing the standard presentation of the braid group \cite{Lonne}.  From this presentation, we can see that $\pi_1(\mathcal{U}_{n,d})$ has a totally symmetric set that is analogous to the standard totally symmetric set in the braid group (see Section~\ref{sec:braid}), and has cardinality $\lceil \frac{d-1}{2}\rceil^n$. Just like in the braid group, when $d \geq 5$ this set collapses if and only if $\rho$ has cyclic image, so Theorem~\ref{thm:linear_tss} gives the desired bound.  (In fact, a slightly better bound can be obtained from a related result in the paper of the first author and Salter \cite[Theorem A]{Caplinger-Salter} that classifies commutative totally symmetric sets in $\GL_n(\C)$.)

\p{Totally symmetric configurations} As per Step 2 of the blueprint, we now explain the classification of large totally symmetric configurations used in the classification of large totally symmetric sets in $\GL_n(\C)$.  

The first author and Salter give the following classification \cite[Theorem A]{Caplinger-Salter}.  In the statement, a standard totally symmetric configuration is the collection of 1-dimensional eigenspaces of the elements of the standard totally symmetric set, or the image of this configuration under any element of $\GL_n(\C)$.  The dual of such a configuration is the set of orthogonal complements (in the paper by the first author and Salter these are referred to as simplex configurations and their duals).

\begin{proposition}[Caplinger--Salter]
\label{prop:arrangements}
Let $\mathcal{W}$ be a totally symmetric configuration of subspaces in $\C^n$. Then $|\mathcal{W}| \leq n+1$, and when $n \neq 5$, equality is realized only by a standard configuration or the dual of such.
\end{proposition}

In the work of the first author and Salter, Proposition~\ref{prop:arrangements} is proved inductively, in tandem with Theorem~\ref{thm:linear_tss}.  Here, we assume the proposition without proof, and show how the inductive step for Theorem~\ref{thm:linear_tss} proceeds.

\p{Proof of Theorem~\ref{thm:linear_tss} assuming Proposition~\ref{prop:arrangements}} We discuss the two statements in turn, namely, the upper bound on the size of a totally symmetric set and the classification of large totally symmetric sets.  

We proceed by induction on $n$, with base case $n=1$.  In this base case we are considering $\GL_1(\C) \cong \C^\ast$.  Since the latter is abelian, the largest totally symmetric set is a singleton.  As $1 \leq 2$ the first statement of the proposition is verified.  The second statement is vacuous since the upper bound $n+1$ is not realized in this case.  

Let $X = \{A_1,\ldots,A_k\}\subset \GL_n(\C)$ be a totally symmetric set. Consider the generalized eigenspaces
\[
E_{\lambda,j}^i = \ker\ (A_i- \lambda I)^j.
\]
If for any $\lambda$ and $j$ the arrangement of subspaces
\[
\{E_{\lambda,j}^1,\ldots E_{\lambda,j}^k\}
\]
is non-degenerate (i.e. not a singleton), then Lemma~\ref{cic} (really, the version for configurations) and Proposition~\ref{prop:arrangements} give the bound $k\leq n+1$. Thus, we may henceforth assume that all arrangements $\{E_{\lambda,j}^i\}_{i=1}^k$ are degenerate.  In other words, the $A_i$ share a common Jordan filtration
\[
E_{\lambda,1}^i \subset E_{\lambda,2}^i \subset \cdots
\]
for every eigenvalue $\lambda$. We may therefore drop the superscript and write $E_{\lambda,j}$ for $E_{\lambda,j}^i
$.

Restricting the $A_i$ to any $E_{\lambda,j}$ gives a totally symmetric set in some $\GL_d(\Z)$ with $d < n$. If this restricted totally symmetric set is non-degenerate, 
then by induction we have $k \leq d+1 < n+1$, as desired.  
Similarly, the maps induced by $A_i$ on the quotients $\C^n / E_{\lambda,j}$ are totally symmetric, and if they are nondegenerate, induction applies. 

We are thus left with the case where all restrictions and quotients associated to each $E_{\lambda,j}$ are identical.  This is a strong condition. From here, the first author and Salter use a variety of techniques to coax out totally symmetric sets of smaller dimension.  They then apply induction to obtain the bound $k \leq n$. 

\medskip

We now turn to the second statement of the theorem, the classification of totally symmetric sets of size $n+1$.  The above argument shows that if $k = n+1$, there must be an eigenvalue $\lambda$ and an index $j$ so that the eigenspace arrangement $\{E_{\lambda,j}^i\}$ corresponding to $X=\{A_i\}$ is nondegenerate.  Moreover, Proposition~\ref{prop:arrangements} implies that this arrangement must be the standard totally symmetric configuration or the dual to such.  

Let $\phi: \Sigma_{n+1} \to \GL_n(\C)$ be a realization map for $X=\{A_i\}$, by which we mean that
\[
\phi(\sigma) A_i \phi(\sigma)^{\inv} = A_{\sigma(i)}
\]
for all $\sigma \in \Sigma_{n+1}$ and all $i \in \{1,\dots,k\}$.  The elements of the image of $\phi$ must permute the set $\{E_{\lambda,j}^i\}_{i=1}^k$ accordingly.

The first author and Salter show that, up to scaling, the only realization map for the standard totally symmetric subspace configuration (or its dual) is the standard representation $\Sigma_{n+1} \to \GL_n(\C)$. Since the standard representation is irreducible, no (proper) arrangement of eigenspaces can be degenerate, as such an arrangement would give an invariant subspace. 

The next step is to show that each $A_i$ has two eigenvalues.  There are two cases, namely, where $\{E_{\lambda,j}^i\}_{i=1}^k$ is the standard totally symmetric configuration and its dual.  To illustrate the idea, we treat the former case.  Assume for the purposes of contradiction that $\lambda$ is the only eigenvalue for the $A_i$.  In this case, $\C^n$ is a single Jordan block, and $\{E_{\lambda,j}\}$ does not stabilize until $j = n$. But then if $n > 2$, the generalized eigenspace $\{E_{\lambda,2}\}$ would be a non-degenerate totally symmetric collection of $2$-dimensional subspaces with $n+1$ elements.  This is impossible by Proposition~\ref{prop:arrangements}. 

By the previous paragraph, we may assume that each $A_i$ has two distinct eigenvalues, $\lambda$ and $\mu$.  The standard representation $\phi:\Sigma_{n+1} \to \GL_n(\C)$ must be a realization map for both $\{E_{\lambda,1}\}$ and $\{E_{\mu,1}\}$. Then both the $\lambda$- and $\mu$-eigenspaces of $A_i$ must be stabilized by $\phi(\Stab(i))$, whose fixed subspaces are $D_i$ and $\C v_i$. Since $\Stab(i) \to D_i$ has no subrepresentations, the eigenspaces must be exactly $D_i$ and $\C v_i$.  The theorem follows.  

%%%
%%%
%%%

\section{Totally symmetric sets in braid groups}
\label{sec:braid}

In this section we use the theory of totally symmetric sets to outline a proof of the following result originally due to Dyer--Grossman \cite{Dyer-Grossman}.  Our argument is a modification of the one used by Kordek and the second author \cite{Kordek-Margalit}, simplified for this special case.

For the statement, let $\sigma_1,\dots,\sigma_{n-1}$ denote the standard half-twist generators for the braid group $\B_n$, and let $\epsilon$ be the automorphism of $\B_n$ given by $\sigma_i \mapsto \sigma_i \inv$ for all $i$.  

\begin{theorem}[Dyer--Grossman]
\label{thm:dg}
For $n \geq 3$ the automorphism $\epsilon$ represents the unique nontrivial outer automorphism of $\B_n$. In particular, 
\[
\Aut(\B_n)\cong \B_n/Z(\B_n) \rtimes \mathbb{Z}/2.
\]
\end{theorem}
The original proof of this theorem by Dyer--Grossman has an algebraic flavor.  The proof we outline here, while only valid as stated for $n \geq 5$, uses combinatorial topology and the theory of mapping class groups.  

In what follows we consider $\B_n$ as the mapping class group of the disk $D_n$ with $n$ marked points in the interior (not to be confused with the dihedral group!).  We will use a number of aspects of the theory of mapping class groups, including the Nielsen--Thurston classification theorem, the canonical reduction systems of Birman--Lubotzky--McCarthy, and the change of coordinates principle.  We refer the reader to the book by Farb and the second author of this article for background on these topics \cite{Primer}.  From the point of view of mapping class groups, the Dyer--Grossman theorem can be stated as: every automorphism of $\B_n$ is induced by a homeomorphism of $D_n$.  

The outline for this section mirrors the one for Section~\ref{sec:gl}.  This stands to reason, as we will be following the same blueprint.

\p{A large commutative totally symmetric set} As per Step 1 of the blueprint, we will require the services of a large (commutative) totally symmetric set in $\B_n$.  The desired set is:
\[
\{\sigma_1, \sigma_3, \sigma_5, \dots, \sigma_m \}
\]
where $m$ is $n-1$ or $n-2$, according to whether $n$ is even or odd, respectively.  As such, the cardinality of this set is $\lfloor n/2 \rfloor$.  That this set is a totally symmetric set is an application of the change of coordinates principle from the theory of mapping class groups.  We refer to any $\B_n$-conjugate of this totally symmetric set as a standard totally symmetric set in $\B_n$.  

\p{Crash course in Nielsen--Thurston theory} Briefly, the Nielsen--Thurston classification gives that every braid is periodic, pseudo-Anosov, or reducible.  Periodic braids have powers that are central in $\B_n$; they correspond to rotations of $D_n$.  Each reducible braid preserves a multicurve, that is, the isotopy class of a collection of pairwise disjoint and pairwise non-homotopic simple closed curves in $D_n$.  Any such multicurve is called a reduction system for the braid.  Pseudo-Anosov braids do not preserve any multicurve.  

For a reducible braid $b$, we may restrict $b$ to the complementary components, and inductively apply the classification.  Thus, there exists a reduction system with the property that the associated restrictions are all periodic or pseudo-Anosov.  There is in fact a unique minimal such reduction system, called the canonical reduction system $\CRS(b)$.  We will use the following properties of canonical reduction systems:
\begin{enumerate}
\item for $a,b \in \B_n$ we have $\CRS(aba\inv) = a\CRS(b)$, and
\item if $a$ and $b$ commute then $\CRS(a)$ and $\CRS(b)$ have trivial geometric intersection.
\end{enumerate}
For the proof, the main points to keep in mind are that there is a map
\[
\CRS : \B_n \to \{\text{multicurves in } D_n\},
\]
and that this map satisfies the above two properties.  

\p{Totally symmetric configurations of multicurves} The configurations we will use to address Step 2 of the blueprint are collections of multicurves in the disk $D_n$.  To each element $g$ of $\B_n$ we associate its canonical reduction system $\CRS(g)$.  Then, to a totally symmetric set $X = \{g_1,\dots,g_k\}$ we can associate the collection of multicurves
\[
\{ \CRS(g_1),\dots, \CRS(g_k)\}.
\]
Because $X$ is totally symmetric, this multicurve configuration is totally symmetric in the sense that for any $\sigma \in \Sigma_k$ there is a braid $g_\sigma$ so that
\[
g_\sigma \cdot ( \CRS(g_1),\dots, \CRS(g_k) ) = ( \CRS(g_{\sigma(1)}),\dots, \CRS(g_{\sigma(k)}) )
\]
(using the first property of $\CRS$ above).  When $X$ is a commutative totally symmetric set, the multicurves $\CRS(g_i)$ have trivial intersection pairwise (using the second property); in what follows, we refer to such a collection as a noncrossing multicurve configuration. 

In the original work of Kordek and the second author, they associate a single labeled multicurve to $X$ instead of a configuration of multicurves.  This is equivalent, but we take this point of view to make the analogy with $\GL_n(\C)$ more clear.

\p{Large noncrossing totally symmetric multicurve  configurations} We now turn towards Step~2 of the blueprint. 
 Let us realize $D_n$ at the closed unit disk in the complex plane, with all marked points on the real axis.  For $1 \leq i \leq n-1$, let $c_i$ be the isotopy class of curves corresponding to a round circle surrounding the $i$th and $(i+1)$st marked points.  The curves are chosen precisely so that $\CRS(\sigma_i)$ is equal to $c_i$.  We define the following noncrossing totally symmetric multicurve in $D_n$:
\[
M_n = \{c_1,c_3,\dots\}.
\]
This is the noncrossing totally symmetric multicurve configuration associated to the totally symmetric set $\{\sigma_1,\sigma_3,\dots\}$ in $\B_n$.  

There are two variations on $M_n$ that we will need to consider.  First, we have the dual totally symmetric configuration
\[
M_n^\ast = \{\{c_1\}^\complement,\{c_3\}^\complement,\dots\}
\]
where $\{c_i\}^\complement$ is the complementary configuration to $c_i$ in $M_n$.  In other words, $\{c_i\}^\complement$ is the multicurve whose components are all the curves appearing in $M_n$ except for $c_i$.

Second, when $n$ is odd, we have the totally symmetric labeled multicurve
\[
\widehat M_n = \{\{c_1,d\},\{c_3,d\},\dots\},
\]
where $d$ is represented by the round curve surrounding the first $n-1$ marked points.  Finally, we can combine these two variations in order to obtain 
\[
\widehat M_n^\ast = \{  \{c_1\}^\complement \cup \{d\},  \{c_3\}^\complement \cup \{d\},  \dots   \}.
\]
That all of these multicurve configurations are totally symmetric follows again from the change of coordinates principle.

Kordek and the second author prove \cite[Lemma 2.3]{Kordek-Margalit} that in fact these are the only examples of large noncrossing totally symmetric multicurve configurations in $D_n$.  In the statement, a configuration of multicurves $\{ m_1,\dots,m_k \}$ is degenerate if two $m_i$ are equal.  

\begin{proposition}[Kordek--Margalit]
\label{prop:multi}
Let $M = \{m_1,\dots,m_k\}$ be a nondegenerate, noncrossing,  totally symmetric multicurve configuration in $D_n$ with $k = \lfloor n/2 \rfloor$. Then $M$ is $\B_n$-equivalent to one of $M_n$, $M_n^\ast$, $\widehat{M}_n$, or $\widehat{M}^\ast_n$.
\end{proposition}

The main idea of the proof is as follows.  Suppose that some $m_i$ contains a curve $c_i^p$ surrounding $p$ marked points.  By total symmetry each of the $m_i$ contains such a curve.  Then if $p > 2$ it must be that these $c_i^p$ are not distinct (otherwise the noncrossing condition would be violated).  We are then led to consider the case that some curve $c$ surrounds exactly $p$ marked points and lies in exactly $d$ of the $m_i$.  Again applying total symmetry, there must be $k \choose d$ such curves, all with pairwise trivial geometric intersection, and all surrounding $p$ marked points.  But for $d < k$, the quantity $k \choose d$ is quadratic in $k = \lfloor n/2 \rfloor$, hence quadratic in $n$.  Again, this violates the noncrossing condition (there are in fact at most $n-2$ pairwise non-isotopic curves in $D_n$ with pairwise trivial geometric intersection).  It follows that the only possibilities are that each curve appearing in an $m_i$ surrounds exactly two marked points, or it lies in all the $m_i$.  From here the proof is straightforward.

\p{Large commutative totally symmetric sets} Continuing with Step 2 of the blueprint, we now explain how Proposition~\ref{prop:multi} is used to classify large commutative totally symmetric sets in $\B_n$.   

To each noncrossing totally symmetric multicurve configuration $M_n$, $M_n^\ast$, $\widehat{M}_n$ and $\widehat{M}^\ast_n$, there is an associated commutative totally symmetric set in $\B_n$, namely:
\begin{align*}
Z_n &= \{ \sigma_1, \sigma_3, \dots \}, \ 
Z_n^\ast = \{ \sigma_1^\ast, \sigma_3^\ast,\dots \}, \\
& \widehat Z_n = \{ \sigma_1 T_d, \sigma_3 T_d, \dots \}, \text{ and } 
\widehat Z_n^\ast = \{ \sigma_1^\ast T_d, \sigma_3^\ast T_d,\dots \}. 
\end{align*}
Here $\sigma_i^\ast$ is the product of the elements of $Z_n$ not equal to $\sigma_i$ and $T_d$ is the Dehn twist about the curve used in the definitions of $\widehat M_n$ and $\widehat{M}_n^\ast$.  We refer to any $\B_n$-conjugate of any of these as a standard commutative totally symmetric set in $\B_n$.  

We can modify any of the standard totally symmetric sets by raising all elements to the same nonzero power.  We can also modify them by multiplying all elements of the set by the same power of $z$, a generator for the (cyclic) center of $\B_n$.  We refer to any totally symmetric set obtained in this way as a modification of a standard commutative totally symmetric set in $\B_n$.  Kordek and the second author prove the following \cite[Lemma 2.6]{Kordek-Margalit}.

\begin{proposition}
\label{prop:bn}
Let $n \geq 3$ and let $k = \lfloor n/2 \rfloor$.  Any commutative totally symmetric set $X$ in $\B_n$ with $|X|=k$ is a modification of a standard one. 
\end{proposition}

To prove the proposition, we of course first use the fact that, through canonical reduction systems, $X$ gives rise to a noncrossing totally symmetric multicurve configuration.  We then use Proposition~\ref{prop:multi} to reduce to the four cases of multicurve configurations given there.

We then treat the four cases in turn.  For the case $\CRS(X) = M_n$, the idea is as follows.  Say that $X$ is $\{g_1,\dots,g_k\}$.  Up to conjugation in $\B_n$ we may assume that $\CRS(g_i)$ is equal to $c_{2i-1}$.  Note that this is the canonical reduction system of $\sigma_{2i-1}$, the $i$th element of $Z_n$.  

We would like to show that $X$ is equal to $Z_n$.  On the exterior of $c_1$ the element $g_1$ is either the identity, periodic, or pseudo-Anosov.  But since $g_1$ commutes with the other $g_i$, it must fix the curves $c_3, c_5,\dots$    There is no (nontrivial) periodic or pseudo-Anosov map that can fix these curves.  It follows that this exterior component of $g_1$ is trivial, and hence that $g_1$ is equal to $\sigma_1^\ell z^s$ for some nonzero $\ell$ and some $s$.  The other three cases are similar.

\p{Proof of Theorem~\ref{thm:dg} assuming Proposition~\ref{prop:bn}} Let $\rho : \B_n \to \B_n$ be an automorphism.  By Lemma~\ref{cic}, the image of $Z_n = \{ \sigma_1, \sigma_3 , \dots\}$ is either a singleton or a commutative totally symmetric set of the same size.  In the first case, it follows that $\rho$ has cyclic image.  Indeed, for $n \geq 5$ the normal closure of $\sigma_1\sigma_3^{-1}$ is $\B_n'$ and $B_n/\B_n' \cong \Z$.  Thus, we may henceforth assume that $\rho(Z_n)$ is a commutative totally symmetric set of cardinality $|Z_n|$. 

By Proposition~\ref{prop:bn}, $\rho(Z_n)$ is conjugate to a modification of $Z_n$, $Z_n^\ast$, $\widehat Z_n$, or $\widehat Z_n^\ast$.  Let us consider the first case.  Up to conjugating $\rho$, we may assume that $\rho(Z_n)$ is exactly a modification  of $Z_n$, that is,
\[
\rho(\sigma_i) = \sigma_i^\ell z^s
\]
for all odd $i$.  

For $i$ even, we then have that $\rho(\sigma_i)$ is conjugate to $\sigma_i^\ell z^s$.  So each such $\rho(\sigma_i)$ is equal to $H_{a_i}^\ell z^s$, where $H_{a_i}$ is the half-twist about a curve $a_i$.  

It is a fact that if $H_a$ and $H_b$ are the half-twists about curves $a$ and $b$ in $D_n$, and they satisfy the braid relation
\[
H_a^\ell H_b^\ell H_a^\ell = H_b^\ell H_a^\ell H_b^\ell,
\]
then $i(a,b)=2$ and $\ell = \pm 1$; see \cite[Lemma 4.9]{bellmargalit}.  Up to the exceptional  automorphism $\epsilon$, we may assume that $\ell=1$.  It then further follows that $s=0$, since an automorphism of $\B_n$ must preserve word length, that is, it respects the abelianization $\B_n \to \Z$. 

It also follows that the sequence of curves
\[
c_1, a_2, c_3, a_4, \dots
\]
is a chain, meaning that consecutive curves intersect twice and all other pairs of curves have trivial geometric intersection.  Up to automorphisms of $\B_n$, we then have (by change of coordinates)
\[
\rho(\sigma_i) = \sigma_i
\]
for all $i$.  In other words, up to modifying $\rho$ by automorphisms, it is the identity.  This completes the proof in the first case.   Using similar reasoning, we rule out the other three possibilities for $\rho(Z_n)$, completing the proof.

\p{From braid groups to mapping class groups} Chen--Mukherjea \cite{Chen-Mukherjea} use a similar approach to classify homomorphisms from the braid group $\B_n$ to the mapping class group $\Mod(S_g)$ when $g < n-2$.  As a corollary, they partially recover the result of Aramayona--Souto \cite{AS} classifying homomorphisms $\Mod(S_g) \to \Mod(S_h)$ for $h < 2g$.

%%%
%%%
%%%

\section{Finite quotients of braid groups and mapping class groups}
\label{sec:kolay}

In 1947, Emil Artin \cite{artinbp} proved that for $n \geq 5$ every non-cyclic homomorphism $\B_n \to \Sigma_n$ is standard.  This means that up to conjugacy, the map sends $\sigma_i$ to the transposition $(i \ \ i+1)$ for all $i$.  His proof uses Bertrand's postulate, a deep fact from number theory which states that every interval $[n,2n]$ contains a prime number.  Artin wrote: ``it would be preferable if a proof could be found that does not make use of this fact.''

Kolay \cite{kolay} found in 2021 a short, elementary proof of Artin's theorem, and in fact proved more.  In the statement, we say that a quotient map is minimal if there is no quotient map whose codomain has smaller cardinality.

\begin{theorem}[Kolay]
\label{thm:kolay}
Let $n \geq 3$.  Up to conjugacy, there is a unique minimal non-cyclic quotient of $\B_n$, namely, the standard map $\B_n \to \Sigma_n$ for $n \neq 4$ or the standard map $\B_4 \to \Sigma_3$.
\end{theorem}

In 2019 the second author of this paper had asked: \emph{What is the smallest non-cyclic quotient of $\B_n$?  Is it $\Sigma_n$?}  Kolay's theorem answers this in the affirmative.

We give Kolay's stunningly simple proof below.  The main ingredients are (1) a large collapsing set in $\B_n$ and (2) the orbit-stabilizer theorem.  While Artin never defined collapsing sets, he certainly had all of the tools to prove Kolay's theorem.  It is remarkable that 74 years passed in between the two works.

Before Kolay's work, partial answers to the second author's question were given by Chudnovsky--Kordek--Li--Partin \cite{reu}, Caplinger--Kordek \cite{caplingerkordek}, and Scherich--Verberne \cite{sv}.

\subsection{Braid groups} In this section we explain Kolay's proof of Theorem~\ref{thm:kolay}, and in the next we explain how Kolay applied the same ideas to the case of the mapping class group.  

\p{Collapsing sets} Let $G$ be a group.  We say that a subset $X \subseteq G$ is a collapsing set if for every group homomorphism $f : G \to H$ the restriction $f|X$ is either injective or constant.  This notion is a generalization of totally symmetric sets.  Indeed, Lemma~\ref{cic} implies that every totally symmetric set is a collapsing set.  

We also remark that under any homomorphism, a collapsing set maps to a collapsing set.  Therefore, there is an analogous blueprint for collapsing sets, an idea that seems to be unexplored.  

\p{Strong collapsing sets} Let $G$ be a group and let $X = \{x_1,\dots,x_k\} \subseteq G$ be a subset.  We say that $X$ is a strong collapsing set if 
\[
G/\normalclosure{x_ix_j^{-1}}
\]
is abelian for all pairs $\{i,j\}$.  This is the same as saying that the normal closure of each $x_ix_j^{-1}$ contains the commutator subgroup $[G,G]$.  If the $x_i$ are all conjugate, the $x_ix_j^{-1}$ lie in $[G,G]$ and so each $G/\normalclosure{x_ix_j^{-1}}$ must exactly be the abelianization of $G$.  It follows from this that a strong collapsing set of conjugate elements is a collapsing set, since conjugate elements in an abelian group are equal.

\p{A large collapsing set} Similar to Step 1 of the totally symmetric set blueprint, we will make use of a large strong collapsing set in $\B_n$ for $n \geq 5$.

For each unordered pair $I = \{i,j\} \subseteq \{1,\dots,n\}$ let $\sigma_I$ denote the half-twist in $\B_n$ given by the counter-clockwise exchange of the $i$th and $j$th marked points in the upper half of $D_n$.  If $I$ and $J$ are distinct ordered pairs then $\sigma_I\sigma_J^{-1}$ is conjugate in $\B_n$ to exactly one of the following: $\sigma_1\sigma_2^{-1}$, $\sigma_1\sigma_3^{-1}$, or $b= \sigma_{\{1,3\}}\sigma_{\{2,4\}}^{-1}$.  

We claim that for $n \geq 5$ the normal closure of any of these three elements in $\B_n$ is the commutator subgroup $\B_n'$. For the first two elements, this is a standard fact.  Similarly, the commutator $[b,\sigma_4]$ both lies in the normal closure of $b$ and is conjugate to $\sigma_1\sigma_2^{-1}$.  It follows that the normal closure of $b$ is again $\B_n'$.  (An alternate, but equivalent, proof of the claim is given by the well-suited arc criterion of Lanier and the second author \cite[Lemma 6.2]{ckm}.)

It follows from the claim that the set of all $\sigma_I$ is a strong  collapsing set for $\B_n$.  We refer to this as the standard strong collapsing set for $\B_n$ and denote it $X_n$.

\p{Two basic group theory facts} In the proof of Theorem~\ref{thm:kolay} we will use the following fact:
\begin{quote}
\emph{(1) If $f : \Z \times H \to G$ is a group homomorphism, and $t$ denotes a generator of $\Z$, then $f(t) \notin f(H)$ if and only if
\[
|f(\Z \times H)| \geq 2 |f(H)|.
\]}
\end{quote}
This fact is true because both conditions are equivalent to the statement that $f(t)f(H)$ is a nontrivial coset of $f(H)$ in $f(\Z \times H)$.  
We will also use the following:
\begin{quote}
\emph{(2) If $G$ is a group, then $Z(G)$ is nontrivial if and only if
\[
|G| \geq 2|G/Z(G)|.
\]}
\end{quote}
This is true because if $Z(G)$ is nontrivial then each coset has at least two elements.  While both statements make sense for infinite groups, we will only apply them when $G$ is finite.

\p{Base cases} We can prove the $n=3$ and $n=4$ cases of Theorem~\ref{thm:kolay} by direct inspection.  Because the abelianization of $\B_n$ is cyclic, a non-cyclic quotient of $\B_n$ is non-abelian.  The only non-abelian group of order 6 or less is $\Sigma_3$.  Thus, all other finite non-abelian quotients of $\B_3$ and $\B_4$ have order strictly greater than 6.  

To see that the standard maps $\B_3 \to \Sigma_3$ and $\B_4 \to \Sigma_3$ are unique up to automorphisms of $\Sigma_3$, we simply check that (up to automorphisms of $\Sigma_3$) the only ordered pair of elements of $\Sigma_3$ satisfying the braid relation is $((1 \ 2),(2 \ 3))$.

\p{Extension of the $n=4$ case} We also will require the following statement:
\begin{quote}
If $f : \B_4 \to G$ is a quotient map that is injective on $X_4$ then $|G| \geq 4!$.  Further, if $|G| = 4!$ then $G = \Sigma_4$ and $f$ is standard.
\end{quote}
Since this can be easily proved with a computer, we omit the proof (although it is a fun exercise to do it by hand!).  

\begin{proof}[Kolay's proof]

We prove the theorem by induction on $n$, with the base cases $n=3$ and $n=4$ (and the extension of the latter) handled as above.  

Let $f : \B_n \to G$ be a non-cyclic quotient map.  Let $X_n$ be the standard strong collapsing set in $\B_n$.  The group $G=f(\B_n)$ acts by conjugation on the conjugacy class of $f(\sigma_1)$ in $G$. To prove that $|G| \geq n!$, we will apply the orbit-stabilizer theorem to this action.  A further analysis will give the statement that $f$ is conjugate to the standard map to $\Sigma_n$. 

For the orbit-stabilizer argument there are, naturally, two steps.  Specifically, if $\mathcal{O} \subseteq G$ and $\mathcal{S} \subseteq G$ are the orbit and stabilizer of $f(\sigma_1)$ then we will show that
\[
|\mathcal{O}| \geq {n \choose 2} \quad \text{and} \quad  |\mathcal{S}| \geq 2 \cdot (n-2)!.
\]

\medskip

\noindent \emph{Orbit.}  Since $X_n$ is a strong collapsing set, and since $G$ is not cyclic, it follows that $|f(X_n)|=|X_n| = {n \choose 2}$.  In particular $|\mathcal{O}| \geq {n \choose 2}$.

\medskip

\noindent \emph{Stabilizer.}  We consider the action of $\B_n$ on itself by conjugation.  The stabilizer of $\sigma_1$ in $\B_n$ contains a subgroup 
\[
\langle \sigma_1,\sigma_3,\sigma_4,\dots,\sigma_{n-1} \rangle \cong \Z \times \B_{n-2}.
\]
The image $f(\Z \times \B_{n-2})$ is a subgroup of $\mathcal{S}$.   We would like to bound the cardinality of this image from below.  We treat two cases, according to whether $f(\sigma_1)$ lies in $f(\B_{n-2})$.

By induction we may assume that $|f(\B_{n-2})| \geq (n-2)!$.  Indeed, if $f(B_{n-2})$ were cyclic, then $f$ would be cyclic, contrary to assumption.  We will use this assumption in both cases.

\bigskip

\noindent Case 1: $f(\sigma_1) \notin f(\B_{n-2})$.  By the first basic group theory fact above, we have
\[
|\mathcal{S}| \geq |f(\Z \times \B_{n-2})| \geq 2 |f(\B_{n-2})| \geq 2 (n-2)!
\]
as desired.

\bigskip

\noindent Case 2: $f(\sigma_1) \in f(\B_{n-2})$. In this case $f(\Z \times \B_{n-2}) = f(\B_{n-2})$.  Since $f$ is nontrivial, $f(\sigma_1)$ is nontrivial.  Since $\sigma_1$ lies in the centralizer of $\B_{n-2}$ it must be that $f(\sigma_1)$ lies in the center of $f(\B_{n-2})$.  In particular, $Z(f(\B_{n-2}))$ is nontrivial.  

We claim that $f(\B_{n-2})/Z(f(\B_{n-2}))$ is not cyclic.  Indeed, if it were cyclic then $f(\B_{n-2})$ would be abelian (for any group $G$, if $G/Z(G)$ is abelian then $G$ is).  The abelianization of $\B_{n-2}$ is cyclic, and so any abelian quotient of it is cyclic.  In particular, $f(\B_{n-2})$ is cyclic.  It then follows that $f(\B_n)$ is cyclic, contrary to assumption.

By the claim, the group $f(\B_{n-2})/Z(f(\B_{n-2}))$ is a non-cyclic quotient of $\B_{n-2}$.  By induction, its order is bounded below by $(n-2)!$.  By the second basic group theory fact, we have
\[
|\mathcal{S}| \geq |f(\Z \times \B_{n-2})| = |f(\B_{n-2})| \geq 2 |f(\B_{n-2})/Z(f(\B_{n-2}))| \geq 2 (n-2)!
\]

We may now complete the proof of the first statement.  By the orbit-stabilizer theorem, we have
\[
|G| \geq f(B_n) \geq |\mathcal{O}||\mathcal{S}| \geq {n \choose 2} 2 \cdot (n-2)! = \frac{n(n-1)}{2} \cdot 2 \cdot (n-2)! = n!
\]

\bigskip

\noindent \emph{The first statement, $n=6$ case.}  For $n=6$ the argument is the same, except we must use the extension of the base case $n=4$.  Since we may assume that $f$ is injective on $X_6$, it is injective on the copy of $X_4$ associated to $\B_4 \leqslant \B_6$.  Hence the size of the stabilizer $\mathcal{S}$ is bounded below by $2 \cdot 4!$.

\bigskip

\noindent \emph{The second statement.} To prove the stronger statement that any quotient of $\B_n$ with order $n!$ is the standard one, it suffices to show that the $f(\sigma_i)$ have order 2 (because the quotient $\B_n \to \B_n /\normalclosure{\sigma_1^2} \cong \Sigma_n$ is the standard quotient). But this is true because in order to realize the lower bound $|\mathcal{S}| \geq 2 \cdot (n-2)!$ it must be true by induction that $f(\B_{n-2})$ is the standard quotient.  
\end{proof}

\subsection{Mapping class groups} We now turn our attention to the analogue of Theorem~\ref{thm:kolay} for mapping class groups.  The natural action of $\Mod(S_g)$ on $H_1(S_g;\mathbb{F}_2)$ gives rise to a representation
\[
\Mod(S_g) \to \Sp_{2g}(\mathbb{F}_2).
\]
The order of the latter group is
\[
|\Sp_{2g}(\mathbb{F}_2)| = 2^{g^2} \prod_{i=1}^g (2^{2i}-1).
\]
Remarkably, we again have that the most natural small quotient is the smallest.  

\begin{theorem}[Kielak--Pierro]
\label{thm:kp}
Let $g \geq 1$.  Up to conjugacy, there is a unique minimal  non-cyclic quotient of $\Mod(S_g)$, namely, the standard map $\Mod(S_g) \to \Sp_{2g}(\mathbb{F}_2)$.
\end{theorem}

This theorem was conjectured by Zimmermann \cite{Zimmermann} in 2012.  Kielak--Pierro \cite{Kielak-Pierro} proved it in 2019, using the approach established by Baumeister--Kielak--Pierro \cite{Baumeister-Kielak-Pierro} in their work on the analogous problem about outer automorphisms of free groups.  

The Kielak--Pierro proof of Theorem~\ref{thm:kp} relies on the classification of finite simple groups and the representation theory of the mapping class group, as well as the deep work of Berrick--Gebhardt--Paris \cite{bgp}, which itself uses the Matsumoto presentation of the mapping class group.  It is astonishing that Kolay's argument for the braid group applies with little modification to prove the same theorem.  

Again, the keys to Kolay's proof of Theorem~\ref{thm:kp} are the construction of a large strong collapsing set in $\Mod(S_g)$, and an orbit-stabilizer argument.

The general outline is closely analogous to the argument for the braid group.  Even the description of the large collapsing set constructed in Step 1 is similar.  Several new tools are required.  We introduce these in turn as we go.  

\p{Reduction to the open case} Let $S_g^1$ denote the surface with one boundary component obtained from $S_g$ by removing the interior of an embedded disk.  Since $H_1(S_g^1;\mathbb{F}_2)$ is naturally isomorphic to $H_1(S_g;\mathbb{F}_2)$ we also have a natural map
\[
\Mod(S_g^1) \to \Sp_{2g}(\mathbb{F}_2).
\]
By filling the disk back in, we also obtain a quotient map
\[
\Mod(S_g^1) \to \Mod(S_g).
\]
In general if (within some class of groups) $G$ is the smallest quotient of a group $M_1$, and $M$ is a quotient of $M_1$ that also has $G$ as a quotient, then $G$ is the smallest quotient of $M$ (in that class of groups).  Thus to prove Theorem~\ref{thm:kp} it suffices to prove the analogous statement for $\Mod(S_g^1)$.

\p{The base case} The base case for the $\Mod(S_g^1)$-version of Theorem~\ref{thm:kp} is the case $g=1$.  In this case Dehn \cite[p. 172]{dehn} proved that $\Mod(S_g^1) \cong \B_3$.  By Theorem~\ref{thm:kolay}, the smallest non-cyclic quotient of this group is $\Sigma_3 \cong \Sp_2(\mathbb{F}_2)$, as desired.  (One way to prove the last isomorphism is to use the formula for the cardinality of $\Sp_2(\mathbb{F}_2)$ and apply Theorem~\ref{thm:kolay}!).  

\p{A well-suited curve criterion} We now turn our attention to Step 1.  In the braid group case, the construction of the large strong collapsing set used the well-known fact that for $i \neq j$ the quotient
\[
\B_n\,/\normalclosure{\sigma_i\sigma_j^{-1}}
\]
is cyclic.  To mimic this step, we will use the following fact:
\begin{quote}
\emph{Let $f \in \Mod(S_g^1)$ and suppose $c$ is a curve with $i(c,f(c))=1$.  Then the normal closure $\normalclosure{f}$ contains the commutator subgroup $\Mod(S_g^1)'$ and
\[
\Mod(S_g^1)/\normalclosure{f}
\]
is cyclic.}
\end{quote}
This fact is an instance of the well-suited curve criterion of Lanier and the second author \cite[Lemma 2.1]{Justin-Dan}.  While their argument is given explicitly for $\Mod(S_g)$, it applies verbatim for $\Mod(S_g^1)$.  The key points are that---like $\Mod(S_g)$---the abelianization of $\Mod(S_g^1)$ is cyclic for $g \geq 1$ and---like $\Mod(S_g)$--- the group $\Mod(S_g^1)$ has a generating set consisting of Dehn twists about curves that have pairwise intersection at most 1.  (In the construction of the large strong collapsing set for $\B_n$ we noted that we could have used the well-suited arc criterion in the proof.  Similarly, it is true here that we can give a proof that mimics the braid group case more closely.  We leave it to the reader to decide which approach they prefer.)

\p{The hyperelliptic involution and mod 2 homology} A hyperelliptic involution of $S_g^1$ is a homeomorphism $\iota$ of order two with $2g+1$ fixed points.  The quotient $S_g^1 / \langle \iota \rangle$ is $D_{2g+1}$, the disk with $2g+1$ marked points.  These marked points are the images of the fixed points of $\iota$.  We denote the set of marked points by $P$.

Let $D_{2g+1}^\circ$ denote the disk with $2g+1$ punctures obtained by removing $P$.  The homology group $H_1(D_{2g+1}^\circ) \cong (\mathbb{F}_2)^{2g+1}$ has a canonical generating set, namely, the classes represented by small loops around the punctures.  This gives rise to a canonical homomorphism
\[
H_1(D_{2g+1}^\circ) \to \mathbb{F}_2,
\]
whereby each of these generators maps to 1.  We denote the kernel of this map by $H_1(D_{2g+1}^\circ;\mathbb{F}_2)^{\text{even}}$.  The elements of this kernel are exactly the ones represented by simple closed curves in $D_{2g+1}$ surrounding an even number of marked points.

We will define a map 
\[
\Psi : H_1(S_g^1;\mathbb{F}_2) \to H_1(D_{2g+1}^\circ;\mathbb{F}_2)
\]
as follows.  Given $v \in H_1(S_g^1;\mathbb{F}_2)$, we may represent $v$ by a simple closed curve $c$ that avoids the fixed points of $\iota$.  The image of $c$ in $D_{2g+1}$ represents an element of $H_1(D_{2g+1}^\circ;\mathbb{F}_2)$.

Arnol'd \cite{arnold} gave the following (easy-to-prove but) remarkable fact:
\begin{quote}
\emph{
The map $\Psi$ is an isomorphism
\[
\Psi : H_1(S_g^1;\mathbb{F}_2) \stackrel{\cong}{\to} H_1(D_{2g+1}^\circ;\mathbb{F}_2)^{\text{even}}
\]
}
\end{quote}
The map $\Psi^{-1}$ can be described as follows.  Given an element $v$ of $H_1(D_{2g+1}^\circ;\mathbb{F}_2)^{\text{even}}$ we represent it by a simple closed curve $c$, and $\Psi^{-1}(v)$ is the class represented by one component of the preimage of $c$.  

Say that a subset of $P$ is even if it has an even number of elements.  By the above discussion we have natural bijections
\[
H_1(D_{2g+1}^\circ;\mathbb{F}_2)^{\text{even}} \leftrightarrow H_1(S_g^1;\mathbb{F}_2) \leftrightarrow \{\text{even subsets of } P \}
\]
Further, the nonzero elements of $H_1(S_g^1;\mathbb{F}_2)$ correspond to the nonempty even subsets of $P$.

\p{The large collapsing set} Let us represent $D_{2g+1}$ as a disk with the points of $P$ lying on a circle.  For each nonempty subset $A \subseteq P$ there is, up to isotopy, a unique curve $c_A$ in $D_{2g+1}$ that bounds a convex disk containing exactly the points of $P$ contained in $A$.  

If $A$ is even then the preimage of $c_A$ in $S_g^1$ has exactly two components.  We choose one of these (arbitrarily) and call it $\tilde c_A$.  

We will show that the set of Dehn twists
\[
X_g = \{ T_{\tilde c_A} \mid \ A \subseteq P \text{ even}, \ \ A \neq \emptyset\, \}
\]
is a strong collapsing set in $\Mod(S_g^1)$.   

In order to prove this, we take $A$ and $B$ to be distinct nonempty even subsets of $P$, and we assume that 
\[
f : \Mod(S_g^1) \to G
\]
is a homomorphism with 
\[
f(T_{\tilde c_A}) = f(T_{\tilde c_B}).
\]
For any choices of $A$ and $B$, there exists an arc $c$ in $D_{2g+1}$ that connects two marked points, that intersects $c_A$ in one point, and that is disjoint from $c_B$.  To check this, we consider two cases, according to whether or not $A \cup B$ is a proper subset of $P$ or not.  In the first case, let $p \in P \setminus (A \cup B)$; we take $c$ to connect $p$ to a point $q \in A$.  If $A \cup B = P$, then since $A$ and $B$ are even there is a point $p$ in $A \cap B$, and $c$ connects any such $p$ to a point $q \in A \setminus B$.  

The preimage of $c$ in $S_g^1$ is a simple closed curve $\tilde c$ with $i(\tilde c,\tilde c_A)=1$ and $i(\tilde c,\tilde c_B)=0$.  It follows that
\[
i(\tilde c,T_{\tilde c_A} T_{\tilde c_B}^{-1}(\tilde c)) = i(\tilde c,T_{\tilde c_A}(\tilde c)) = 1.
\]
By the above well-suited curve criterion, we conclude that $X_g$ is a strong collapsing set, as desired.

\p{A derivative collapsing set} Say that two elements of the strong collapsing set $X_g$ are dual if the corresponding curves have intersection number 1 (equivalently if they have algebraic intersection number 1).  We define 
\[
X_g^\prime = \{(x,y) \in X_g \times X_g \mid x \text{ is dual to } y\}.
\]
The set $X_g^\prime$ is a collapsing set in the following sense: if $f : \Mod(S_g^1) \to G$ is a non-cyclic homomorphism, then each $f(x,y)$ is an ordered pair of distinct elements, and if the $f$-image of any two elements of $X_g^\prime$ coincide, they all coincide.  Both of these statements follow from the fact that $X_g$ is a collapsing set.  

Since $X_g$ is in bijection with the nonzero elements of $H_1(S_g^1;\mathbb{F}_2)$, it has $2^{2g}-1$ elements.  Given one nonzero element of $H_1(S_g^1;\mathbb{F}_2)$, there is a codimension-1 affine subspace of $H_1(S_g^1;\mathbb{F}_2)$ corresponding to dual elements in $H_1(S_g^1;\mathbb{F}_2)$.  Thus we have
\[
|X_g^\prime| = (2^{2g}-1) \cdot 2^{2g-1}
\]
We are finally ready for the proof of the Kielak--Pierro theorem.

\p{Proof of Theorem~\ref{thm:kp}} We proceed by induction on $g$.  We already checked the base case $g=1$.  Assume then that the theorem holds for $\Mod(S_g^1)$ with $g \geq 1$.  We will show that it holds for $\Mod(S_{g+1}^1)$.  

Let $f : \Mod(S_{g+1}^1) \to G$ be a non-cyclic quotient.  We use essentially the same orbit-stabilizer argument that was used in the proof of Theorem~\ref{thm:kolay}.  The group $G$ acts by conjugation on the set of ordered pairs of elements of $G$.  Because $X_{g+1}^\prime$ is a collapsing set in the sense described above, the orbit is bounded below by
\[
|f(X_{g+1}^\prime)|=|X_{g+1}^\prime| = (2^{2(g+1)}-1) \cdot 2^{2(g+1)-1} = (2^{2(g+1)}-1) \cdot 2^{2g+1}
\]
The stabilizer contains a copy of the image of $\Mod(S_g^1)$.  By induction this gives a lower bound of
\[
|f(\Mod(S_g^1)| \geq 2^{g^2} \prod_{i=1}^g (2^{2i}-1).
\]
Multiplying these together gives the desired bound
\[
2^{(g+1)^2} \prod_{i=1}^{g+1} (2^{2i}-1).
\]
For this bound to be realized, the elements of $X_g$ must map to elements of order 2.  From there it follows that the quotient is $\Sp_{2(g+1)}(\mathbb{F}_2)$.  This completes the proof.

\p{A final lament} Kolay's proof of Theorem~\ref{thm:kp} goes through the derived collapsing set $X_g'$.  An analogous argument can be used to give an unnecessarily complicated proof of Theorem~\ref{thm:kolay} (about braid groups).  The situation suggests to the authors that there should be a proof of Theorem~\ref{thm:kp} that uses $X_g$ directly, and decreases genus in two inductive steps.  We were not able to find such a proof.  We implore the reader to find one.

%%%
%%%
%%%

\section{Speculations and representations}
\label{sec:spec}

As suggested to us by Kordek, there is a strong analogy between the collision-implies-collapse property and Schur's lemma from representation theory.  We can give weight to this analogy as follows.  

A homomorphism $f : G \to H$ induces a linear map $f_\ast : \C[G] \to \C[H]$.  The vector spaces $\C[G]$ and $\C[H]$ come equipped with a $G$-action and an $H$-action, respectively, where both groups act by conjugation on the basis elements.  If $X \subseteq G$ is a totally symmetric set with $|X|=k$, and $G_X$ is the stabilizer of $X$ in $G$, then $\C[X]$ is a representation of $G_X$.  By total symmetry $G_X$ surjects onto $\Sigma_k$, and so $\C[X]$ is a representation of $G_X$; in fact this representation factors through the permutation representation of $\Sigma_k$ on $\C[X]$.  On the other hand, the vectors space $\C[f(X)]$ is a representation of $f(G_X) \subseteq H_{f(X)}$.  Lemma~\ref{cic} implies that the latter representation has either the same dimension as $\C[X]$ or it has dimension 1.  This statement can be derived from Schur's lemma using the following three facts: (1) the first representation factors through $\Sigma_k$, (2) $f_\ast$ intertwines the two representations, and (3) the permutation representation of $\Sigma_k$ is the direct product of two irreducible representations of $\Sigma_k$, namely, the standard representation and the trivial one.

\p{Other groups} Because total symmetry can be understood within representation theory as above, we are led to speculate on which aspects of representation theory can be brought to bear in the theory of totally symmetric sets.  To begin, we know that, while the  representation theory of the symmetric group is rich in and of itself, there is a broad landscape of representations of various groups.

\begin{question}
To what extent, and to what end, can the theory of totally symmetric sets be generalized to arbitrary groups besides $\Sigma_k$?  \end{question}
  
 As an example of what we have in mind, we note that there is no lift of the totally symmetric set $\{(1\ i)\} \subseteq \Sigma_n$ to a totally symmetric set in $\B_n$.  On the other hand, there is a lift to a cyclically symmetric set, that is, a set with an action of $\Z/(n-1)$.  Similarly, there are large sets of Dehn twists in $\Mod(S_g)$ that carry an action by the dihedral group $D_{2g}$.  How can these sets be used in the classification of homomorphisms between braid groups and mapping class groups?
 
 \p{Extending the analogy to representation theory} Because of the connection between totally symmetric sets and representation theory described above, it is natural to ask which notions from representation theory have analogues for totally symmetric sets.  
 
\begin{question}
Which of the concepts in representation theory---direct sum, direct product, tensor product, etc.---have analogues in the theory of totally symmetric sets?  
\end{question}
 
Already in their work, Salter and the first author give versions of subrepresentations and induced representations for totally symmetric sets \cite{Caplinger-Salter}.  
 
\p{Multiple totally symmetric sets} The arguments presented in this paper are carried out by analyzing the action of a homomorphism on a single totally symmetric set or collapsing set. But many groups, such as the braid group, contain totally symmetric sets that are compatible in some sense (for instance, elements either commute or braid).  
 
\begin{question}
How can multiple totally symmetric sets in a group be used to give stronger constraints on homomorphisms than can be obtained with a single totally symmetric sets?
\end{question}
 
One step in this direction is taken in the work of Scherich--Verberne, where they study homomorphisms of virtual, welded, and classical braid groups by considering multiple totally symmetric sets at once \cite{sv}.  
 
\p{Bounds on representations} As we have seen, the fact that $\Sigma_n$ has a totally symmetric set of cardinality $n-1$ can be used, along with the work of the first author and Salter, to give a lower bound on the dimension of a non-cyclic linear representation of $\Sigma_n$.  We are curious to what extent this line of reasoning holds for other groups.  

As a first test case, we might consider the monster group $M$.  The smallest nontrivial representation of $M$ has dimension $47 \cdot 59 \cdot 71 = 196,883$.  Based on the case of the symmetric group, one might hope that this is because $M$ contains a totally symmetric set of cardinality 196,883.  But we already showed in Section~\ref{sec:tsc} that $M$ cannot contain a totally symmetric set whose cardinality is greater than~43.  

\begin{question}
What are the largest totally symmetric sets of the monster group? Can they give insight into the 196883-dimensional representation?
\end{question}

It is tantalizing that totally symmetric sets might give new insights into the notoriously mysterious monster group.  And similarly for other groups, discovered and not.

\bibliographystyle{amsplain}
\bibliography{bib}

\end{document}